\documentclass{article}
\usepackage{amsthm,amsmath,amssymb,amsfonts}
  \newtheorem{thm}{Theorem}[section]
  \newtheorem{lem}[thm]{Lemma}
  \newtheorem{prop}[thm]{Proposition}
  \newtheorem{cor}[thm]{Corollary}
  \theoremstyle{definition}
  \newtheorem{defn}[thm]{Definition}
  \newtheorem{exm}[thm]{Example}
  \newtheorem{rmk}[thm]{Remark}
  
 \newcommand\mar{\rightsquigarrow}
 \newcommand\ra{\rightarrow}

 \newcommand\s{\subseteq}
 \newcommand\bsim{\backsim}

 \numberwithin{equation}{section}
\def\iff{if and only if }

\def\iff{if and only if }
\begin{document}

\title{ Pseudo Equality Algebras -- Revision }
\author{ Anatolij Dvure\v{c}enskij$^{^{1,2}}$, Omid Zahiri \\
{\small\em $^1$Mathematical Institute,  Slovak Academy of Sciences,}\\
{\small\em\v Stef\'anikova 49, SK-814 73 Bratislava, Slovakia} \\
{\small\em $^2$Depart. Algebra  Geom.,  Palack\'{y} Univer.17. listopadu 12,}\\
{\small\em CZ-771 46 Olomouc, Czech Republic} \\
{\small\tt  dvurecen@mat.savba.sk\quad   om.zahiri@gmail.com} }
\date{}
\maketitle
\begin{abstract}
Recently Jenei introduced a new structure called equality algebras which is inspired by ideas of BCK-algebras with meet. These algebras were generalized by Jenei and K\'or\'odi to pseudo equality algebras which are aimed to find a connection with pseudo BCK-algebras with meet. We show that every pseudo equality algebra is an equality algebra. Therefore, we define a new type of pseudo equality algebras which more precisely reflects the relation to pseudo BCK-algebras with meet in the sense of Kabzi\'nski and Wro\'nski. We describe congruences via  normal closed deductive systems, and we show that the variety of pseudo equality algebras is subtractive, congruence distributive and congruence permutable.
\end{abstract}



{\small {\it Acknowledgement:} This work was supported by  the Slovak Research and Development Agency under contract APVV-0178-11,  grant VEGA No. 2/0059/12 SAV, and
CZ.1.07/2.3.00/20.0051.}

\section{ Introduction }

In the last decade, pseudo BCK-algebras introduced by Georgescu and Iorgulescu \cite{Gor}, which are a non-commutative generalization of
BCK-algebras, are intensively studied. They are inspired by a non-commutative fuzzy logic. Besides pseudo BCK-algebras, there are also other non-commutative generalizations of known commutative algebras connected with logics: pseudo MV-algebras \cite{GeIo} (equivalently generalized MV-algebras by \cite{Rac}), pseudo BL-algebras, \cite{DGI1,DGI2}, pseudo hoops \cite{GLP}, etc.

Equality algebras were introduced in \cite{J} by Jenei to find something similar to $EQ$-algebras, introduced by Nov\'ak and Baets in \cite{ND}, but without a product. These algebras are assumed for a possible algebraic semantics of fuzzy type theory. An {\it equality algebra} is an algebra $(X;\sim,\wedge,1)$ of type $(2,2,0)$
such that the following axioms are fulfilled for all $a,b,c\in X$:
\begin{itemize}
\item[$(E1)$] $(X;\wedge,1)$ is a meet-semilattice with top element 1;
\item[$(E2)$] $a\sim b=b\sim a$;
\item[$(E3)$] $a\sim a=1$;
\item[$(E4)$] $a\sim 1=a$;
\item[$(E5)$] $a\leq b\leq c$ implies that $a\sim c\leq b\sim c$ and $a\sim c\leq a\sim b$;
\item[$(E6)$] $a\sim b\leq (a\wedge c)\sim (b\wedge c)$;
\item[$(E7)$] $a\sim b\leq (a\sim c)\sim (b\sim c)$.
\end{itemize}
He defined a closure operator in the class of equality algebras, and proved the term
equivalence of the closed algebras (called them equivalential equality algebras) to
$BCK$-meet-semilattices.

Recently, a new algebraic structure called a pseudo equality algebra has been
defined by Jenei and K\'{o}r\'{o}di \cite{JK} as a generalization of equality algebras:

\begin{defn}\label{2.2}\cite{JK}
A pseudo equality algebra is an algebra $(X;\sim,\bsim,\wedge,1)$ of type $(2,2,2,0)$
satisfying the following axioms, for all $a,b,c\in X$:
\begin{itemize}
\item[$(F1')$] $(X;\wedge,1)$ is a meet-semilattice with top element 1;
\item[$(F2')$] $a\sim b=b\sim a$ and $a\bsim b=b\bsim a$;
\item[$(F3')$] $a\sim a=1=a\backsim a$;
\item[$(F4')$] $a\sim 1=a=a\backsim 1$;
\item[$(F5')$] $a\leq b\leq c$ implies that $a\sim c\leq b\sim c$, $a\sim c\leq a\sim b$,
$c\backsim a\leq c\backsim b$ and $c\backsim a\leq b\backsim a$;
\item[$(F6')$] $a\sim b\leq (a\wedge c)\sim (b\wedge c)$ and $a\bsim b\leq (a\wedge c)\bsim (b\wedge c)$;
\item[$(F7')$] $a\sim b\leq (a\sim c)\bsim (b\sim c)$ and $a\bsim b\leq (a\bsim c)\sim (b\bsim c)$.
\end{itemize}
\end{defn}
The authors attempted to generalized the results for pseudo equality algebra,
showing that equivalential pseudo equality algebras are term equivalent with
pseudo $BCK$-meet-semilattices and provided an equational characterization for
the equivalence operations of pseudo $BCK$-meet-semilattices which corresponds to ideas by Kabzi\'nski and Wro\'nski \cite{KaWr}. They also, proved that the variety of pseudo equality algebras is a subtractive 1-regular, arithmetical variety. In \cite{Ciu}, Ciungu found a gap in the proof
of a theorem of the paper \cite{JK} and she presented a counterexample and a correct version of it. The correct version of the corresponding result for equality algebras was also given.

The present paper is inspired by our simple observation that every pseudo equality algebra is in fact an equality algebra. Therefore, we introduce a new kind of pseudo equality algebras which will fit all ideas of Kabzi\'nski and Wro\'nski as well as of Jenei and K\'or\'odi. Our results may be assumed as an additional step in establishing an axiomatization of special subclasses of substructural logics.

We show an intimate relation of pseudo  equality algebras with a special class of pseudo equality algebras with meet. We describe the variety of pseudo equality algebras, we present the congruence lattice and we show how congruences are closely bind with closed normal deductive systems.

The paper is organized as follows. Section 2 gives a new definition of pseudo equality algebras. We show some examples a describe some important properties of the algebras. Section 3 shows a relation between pseudo equality algebras and pseudo BCK-algebras with meet. In Section 4, we describe congruences and deductive systems. We show when we can create a quotient. Finally, we prove that the variety of pseudo equality algebras is subtractive, congruence permutable and congruence distributive.

\section{New definition of pseudo equality algebras}

In the section, we show that any pseudo equality algebra in the sense of  \cite{JK} is always an equality algebra. Therefore, we introduce a new type of pseudo equality algebras and we describe their basic properties.

Consider Definition \ref{2.2}. If $(A;\sim,\backsim,\wedge,1)$ is a pseudo equality algebra,
then $(A;\wedge)$ is a
meet semi-lattice and so the relation $x\leq y$ if and only if $x\wedge y=x$ is a partially order relation on $A$.
By $(F7')$, for all $a,b,c\in A$, we have
$$
a\sim b\leq (a\sim c)\backsim (b\sim c), \quad a\backsim b\leq (a\backsim c)\sim (b\backsim c).
$$
Specially, for $c=1$, we get that
$a\sim b\leq (a\sim 1)\backsim (b\sim 1)$ and $a\backsim b\leq (a\backsim 1)\sim (b\backsim 1)$
so by $(F4)$,
$a\sim b\leq (a\sim 1)\backsim (b\sim 1)=a\backsim b$ and $a\backsim b\leq (a\backsim 1)\sim (b\backsim 1)=a\sim b$.
It follows that $a\sim b=a\backsim b$, for all $a,b\in A$ and so $(A;\wedge,\sim,\backsim,1)$
is an equality algebra. Therefore, {\it every pseudo equality algebra is  an equality algebra} in the sense of Definition \ref{2.2}.

In the next definition, we will propose a new definition of pseudo equality algebras that will not imply that they are in fact equality algebras. They are inspired by some properties of pseudo BCK-algebras that are also a $\wedge$-semilattice.

\begin{defn}\label{3.1}
A {\it pseudo equality algebra} is an algebra $(X;\sim,\bsim,\wedge,1)$ of type $(2,2,2,0)$ that
satisfies the following axioms, for all $a,b,c\in X$:
\begin{itemize}
\item[$(F1)$] $(X;\wedge,1)$ is a meet-semilattice with top element 1;
\item[$(F2)$] $a\sim a=1=a\backsim a$;
\item[$(F3)$] $a\sim 1=a=1\backsim a$;
\item[$(F4)$] $a\leq b\leq c$ implies that $a\sim c\leq b\sim c$, $a\sim c\leq a\sim b$,
$c\backsim a\leq c\backsim b$ and $c\backsim a\leq b\backsim a$;
\item[$(F5)$] $a\sim b\leq (a\wedge c)\sim (b\wedge c)$ and $a\bsim b\leq (a\wedge c)\bsim (b\wedge c)$;
\item[$(F6)$] $a\sim b\leq (c\sim a)\bsim (c\sim b)$ and $a\bsim b\leq (a\bsim c)\sim (b\bsim c)$;
\item[$(F7)$] $a\sim b\leq (a\sim c)\sim (b\sim c)$ and $a\bsim b\leq (c\bsim a)\bsim (c\bsim b)$.
\end{itemize}
\end{defn}

\begin{rmk}\label{re:3.2}
(1) Property $(F4)$ can be rewritten in the following form of equations:
\begin{itemize}
\item[$(F4')$] $(a\wedge b \wedge c)\sim c\le (b\wedge c)\sim c$, $(a\wedge b \wedge c)\sim c \le (a\wedge b \wedge c)\sim (b\wedge c)$, $c \backsim (a\wedge b \wedge c) \le c \backsim (b\wedge c)$ and $c\backsim (a\wedge b \wedge c) \le (b\wedge c)\backsim (a\wedge b \wedge c)$.
\end{itemize}
Consequently, the class of pseudo equality algebras forms a variety.

(2) We note that $(F5)$ implies the second and fourth inequality in $(F4)$.

(3) In addition, if $(X;\sim,\wedge,1)$ is an equality algebra, then $(X;\sim,\sim,\wedge, 1)$ is a pseudo equality algebra.

\end{rmk}

\begin{rmk}\label{3.2}
If $(X;\sim,\bsim,\wedge,1)$ is a pseudo equality algebra such that $\sim$ and $\bsim$ are commutative
binary operations on $X$, then by $(F2)$ and $(F6)$, we have
$a\sim b\leq (1\sim a)\bsim (1\sim b)=a\bsim b$ and $a\bsim b\leq (1\bsim a)\sim (1\bsim b)=a\bsim b$
for all $a,b\in X$. Hence $\sim=\bsim$ and $(X;\sim,\wedge,1)$ is an equality algebra.
\end{rmk}
\begin{cor}\label{3.6}
If $(X;\sim,\bsim,\wedge,1)$ is a pseudo equality algebra such that $1\sim a=a=a\bsim 1$
for all $a\in X$, then
 $\sim=\bsim$ and $(X;\sim,\wedge,1)$ is an equality algebra.
\end{cor}
\begin{proof}
Let $a,b\in X$. By $(F6)$,
$a\sim b\leq (b\sim a)\bsim (b\sim b)=(b\sim a)\bsim 1=b\sim a$, similarly,
$b\sim a\leq a\sim b$ and so $a\sim b=b\sim a$. In a similar way, we can show that
$a\bsim b=b\bsim a$ and so by Remark \ref{3.2}, $(X;\sim,\wedge,1)$ is an equality algebra.
\end{proof}

Now we show two classes of pseudo equality algebras. The first one is connected with the negative cones of $\ell$-groups.

\begin{exm}\label{ex:l-group}
Let $(G;\cdot,^{-1},e,\le)$ be an $\ell$-group (= lattice ordered group) written multiplicatively with an inversion $^{-1}$ and the identity element $e$, equipped with a lattice order $\le$ such that $a\le b$ entails $cad\le cbd$ for all $c,d \in G$. We denote $G^+=\{g \in G: e\le g\}$ and $G^-=\{g \in G: g \le e\}$  the positive and negative cone, respectively, of $G$. If we endow the negative cone $G^-$ with two binary operations $a\sim b = (a b^{-1})\wedge e$, $a\backsim b = (a^{-1} b)\wedge e$,  then $(G^-;\sim,\backsim,\wedge, e)$ is an example of a pseudo hoop (see also the next example). We have $a\sim b = b\backsim a$ \iff $G$ is Abelian.
\end{exm}

The second class is more general and is connected with pseudo hoops that were presented in \cite{GLP} and which were originally introduced by Bosbach in \cite{Bos1, Bos2} under the name ``residuated integral monoids''.

\begin{exm}\label{phooh}
A {\it pseudo hoop} is an algebra $\mathbf (X;\odot,\ra,\mar,1)$ of type
$(2,2,2,0)$ if the following holds, for all $a,b \in X$:
\begin{enumerate}
\item[(i)] $a \odot 1 = 1\odot a=a$;
\item[(ii)] $a\ra a=1=a\mar a$;
\item[(iii)] $(a\odot b)\ra c = a \ra    (b\ra a)$;
\item[(iv)] $(a\odot b)\mar c= b \mar (a\mar c)$;
\item[(v)] $(a \ra b)\odot a= (b\to a)\odot a= a \odot (a\mar b)=b\odot (b\mar a)$.
\end{enumerate}
Then $X$ is a $\wedge$-semilattice, where $a\wedge b = a \odot(a\mar b)$.

If we set $a\sim b = b\ra a$ and $a\backsim b= a\mar b$, then $(X;\sim,\backsim,\wedge, 0)$ is a pseudo equality algebra. Indeed, $(F2)$ follows from (ii), $(F3)$ from \cite[Lem 2.4(3)(4)]{GLP}, $(F4)$ from \cite[Lem 2.5(12),(13)]{GLP}, $(F6)$ from \cite[Lem 2.5(16),(17)]{GLP}. To prove $(F5)$, we use \cite[Lem 2.7]{GLP}:
\begin{eqnarray*} &a\rightarrow b=(a\rightarrow b)\odot 1=(a\rightarrow b)\odot (c\rightarrow c)\leq (a\wedge c)\rightarrow (b\wedge c),\\
&a\rightsquigarrow b=(a\rightsquigarrow b)\odot 1=(a\rightsquigarrow b)\odot (c\rightsquigarrow c)\leq (a\wedge c)\rightsquigarrow (b\wedge c).
\end{eqnarray*}

$(F7)$ holds due to \cite[Lem 2.4(5),(6)]{GLP}.

\end{exm}

In any pseudo equality algebra $(X;\sim,\bsim,\wedge,1)$ we can define two derived binary operations
on $X$, by $x\ra y:=(x\wedge y)\sim x$ and $x\mar y:=x\bsim(x\wedge y)$ for all $x,y\in X$.
It can be easily shown that Proposition 1 of \cite{JK} is correct in a new definition of pseudo equality algebras:




\begin{prop}\label{3.4}
Let $(X;\sim,\bsim,\wedge,1)$ be a pseudo equality algebra.
\begin{itemize}
\item[{\rm(i)}] $(a\wedge b)\sim a\leq (a\wedge b\wedge c)\sim (a\wedge c)$ and
$a\ra b\leq (a\wedge c)\ra b$ for all $a,b,c\in X$;
\item[{\rm(ii)}] $a\bsim (a\wedge b)\leq (a\wedge c)\bsim (a\wedge b\wedge c)$ and
$a\mar b\leq (a\wedge c)\mar b$ for all $a,b,c\in X$;
\item[{\rm(iii)}] $c\leq b$ implies that $a\ra c\leq a\ra b$ and $a\mar c\leq a\mar b$
for all $a,b,c\in X$;
\item[{\rm(iv)}] $c\leq a$ implies that $a\ra b\leq c\ra b$ and $a\mar b\leq c\mar b$ for all
$a,b,c\in X$.
\end{itemize}
\end{prop}
\begin{proof}
Let $a,b,c\in X$. Then by $(F5)$, $(a\wedge b)\sim a\leq (a\wedge b\wedge c)\sim (a\wedge c)$ and so
$a\ra b\leq (a\wedge c)\ra b$. Moreover,
$a\bsim(a\wedge b)\leq (a\wedge b)\bsim(a\wedge b\wedge c)$ and
$a\mar b\leq (a\wedge c)\mar b$. This completes the proof of parts (i) and (ii).
The proof of (iii) and (iv) is straightforward by  (i) and (ii). 
\end{proof}

\begin{prop}\label{3.5}
Let $(X;\sim,\bsim,\wedge,1)$ be a pseudo equality algebra. Then the following hold for all
$a,b,c\in X$:
\begin{itemize}
\item[{\rm(i)}] $a\bsim b\leq a\mar b$ and $b\sim a\leq a\ra b$;
\item[{\rm(ii)}] $a\leq ((c\sim a)\bsim c)\wedge (c\sim (a\bsim c))$;
\item[{\rm(iii)}] $a\bsim b=1$ or $b\sim a=1$ imply $a\leq b$;
\item[{\rm(iv)}] $a\sim b=1$  implies $c\sim a\leq c\sim b$ and $a\bsim b=1$ implies $b\bsim c\leq a\bsim c$;
\item[{\rm(v)}] $a\leq b$ if and only if $a\ra b=1$ if and only if $a\mar b=1$;
\item[{\rm(vi)}] $a\mar1=a\mar a=a\ra a=a\ra 1=1$, $1\mar a=a$ and  $1\ra a=a$;
\item[{\rm(vii)}] $a\leq (b\ra a)\wedge (b\mar a)$;
\item[{\rm(viii)}] $a\leq ((a\ra b)\mar b)\wedge ((a\mar b)\ra b)$;
\item[{\rm(ix)}] $a\ra b\leq (b\ra c)\mar (a\ra c)$ and $a\mar b\leq (b\mar c)\ra (a\mar c)$;
\item[{\rm(x)}] $a\leq b\ra c$ if and only if $b\leq a\mar c$;
\item[{\rm(xi)}] $a\ra (b\mar c)=b\mar (a\ra c)$;
\item[{\rm(xii)}] $b\ra a\leq (b\wedge c)\ra (a\wedge c)$ and $b\mar a\leq (b\wedge c)\mar (a\wedge c)$;
\item[{\rm(xiii)}] $a\ra b = a\ra (a\wedge b)$ and $a\mar b = a\mar (a\wedge b)$.
\end{itemize}
\end{prop}
\begin{proof}
(i)  By $(F5)$, $b\sim a\leq (b\wedge a)\sim (a\wedge a)=(a\wedge b)\sim a=a\ra b$
and $a\bsim b\leq (a\wedge a)\bsim (b\wedge a)=a\bsim (a\wedge b)=a\mar b$.

(ii) By substituting $1$ for $b$ in $(F6)$, we have
$a=a\sim 1\leq (c\sim a)\bsim (c\sim 1)= (c\sim a)\bsim c$. Also,
$a=1\bsim a\leq (1\bsim c)\sim (a\bsim c)=c\sim (a\bsim c)$.

(iii) Let $a\bsim b=1$. Then by (ii), $a\leq b\sim (a\backsim b)=b \sim 1=b$. Similarly, the case $b\sim a=1$ implies
$a\leq (b\sim a)\backsim b=1\sim b=b$.

(iv) Let $a\sim b=1$. Then by $(F6)$, $1=a\sim b\leq (c\sim a)\bsim (c\sim b)$ and so
by (iii), $c\sim a\leq c\sim b$. Moreover, if $a\bsim b=1$, then
$1=a\bsim b\leq (a\bsim c)\sim (b\bsim c)$, hence by (iii), $b\bsim c\leq a\bsim c$.

(v) Let $a\leq b$. Then $a=a\wedge b$, hence $a\ra b=(a\wedge b)\sim a=a\sim a=1$. Conversely,
$a\ra b=1$ implies that $(a\wedge b)\sim a=1$ and so by (iii), $a\leq a\wedge b$. It follows that
$a=a\wedge b$ and $a\leq b$. The proof of the second part is similar.

(vi) They follow easily from definition of $\mar$ and $\ra$.


(vii) By (vi) and Proposition \ref{3.4}, we have $a=1\ra a\leq b\ra a$ and
$a=1\mar a\leq b\mar a$.

(viii) By (ii) and (i) (respectively),
we get that $a\leq ((a\wedge b)\sim a)\bsim (a\wedge b)=(a\ra b)\bsim (a\wedge b)
\leq (a\ra b)\mar (a\wedge b)$.
Also, by Proposition \ref{3.4}(iii), $(a\ra b)\mar (a\wedge b)\leq (a\ra b)\mar b$, so
$a\leq (a\ra b)\mar b)$. On the other hand, by (ii), (i) and Proposition \ref{3.4}(iii),
$a\leq (a\wedge b)\sim(a\bsim (a\wedge b))=(a\wedge b)\sim (a\mar b)\leq (a\mar b)\ra (a\wedge b)
\leq (a\mar b)\ra b$.

(ix) By $(F5)$, we know that $b\sim (b\wedge c)\leq (a\wedge b)\sim (a\wedge b\wedge c)$. Since
$a\wedge b\wedge c\leq a\wedge c\leq a$, then by $(F4)$, $(a\wedge b\wedge c)\sim a\leq (a\wedge c)\sim a$.
It follows that
\begin{eqnarray}
((a\wedge b\wedge c)\sim (a\wedge b))\mar ((a\wedge b\wedge c)\sim a)&\leq & ((c\wedge b)\sim b)\mar ((a\wedge b\wedge c)\sim a), \mbox{ Prop \ref{3.4}(iv)}\\
\label{R1} &\leq & ((c\wedge b)\sim b)\mar ((a\wedge c)\sim a), \mbox{ Prop \ref{3.4}(iii)}.
\end{eqnarray}
Moreover,
\begin{eqnarray*}
a\ra b=(a\wedge b)\sim a&\leq &((a\wedge b\wedge c)\sim(a\wedge b))\bsim((a\wedge b\wedge c)\sim a),\mbox{ by $(F6)$}\\
&\leq & ((a\wedge b\wedge c)\sim(a\wedge b))\mar((a\wedge b\wedge c)\sim a),\mbox{ by (i)}\\
&\leq & ((c\wedge b)\sim b)\mar ((a\wedge c)\sim a), \mbox{ by (\ref{R1})}\\
&=& (b\ra c)\mar (a\ra c).
\end{eqnarray*}

(x) Let $a\leq b\ra c$. By (viii) and Proposition \ref{3.4}(iv), $b\leq (b\ra c)\mar c\leq a\mar c$. Conversely, let $b\leq a\mar c$. Then
similarly, by (viii) and Proposition \ref{3.4}(iv), $a\leq (a\mar c)\ra c\leq b\ra c$.

(xi) By (viii), $a\leq (a\ra c)\mar c$, hence by Proposition \ref{3.4}(iv), $a\ra (b\mar c)\geq ((a\ra c)\mar c)\ra (b\mar c)$, whence
by (ix), $a\ra (b\mar c)\geq b\mar (a\ra c)$. Also, by (viii), $b\leq (b\mar c)\ra c$, hence
$b\mar (a\ra c)\geq ((b\mar c)\ra c)\mar (a\ra c)\geq a\ra (b\mar c)$. Therefore, $a\ra (b\mar c)=b\mar (a\ra c)$.

(xii) By $(F5)$, we have $b\ra a=(a\wedge b)\sim b\leq ((a\wedge b)\wedge c)\sim (b\wedge c)=
((a\wedge c)\wedge (b\wedge c))\sim (b\wedge c)=(b\wedge c)\ra (a\wedge c)$.

(xiii) They trivially follow from definition of $\ra$ and $\mar$.
\end{proof}

\section{Relation between pseudo equality algebras and pseudo $BCK$-algebras}

We show that similarly as in \cite{JK,Ciu}, there is a close connection between pseudo equality algebras and pseudo BCK-algebras with meet and with a special condition.

\begin{defn}\cite{Gor,Ior}\label{4.1}
A pseudo $BCK$-algebra is an algebra $(X;\ra,\mar,1)$ of type $(2,2,0)$ that satisfies the following axioms:
\item[$(PBCK1)$] $(a\ra b)\mar ((b\ra c)\mar (a\ra c))=1$ and $(a\mar b)\ra ((b\mar c)\ra (a\mar c))=1$;
\item[$(PBCK2)$] $1\ra a=a$ and $1\mar a=a$;
\item[$(PBCK3)$] $a\ra 1=1$ and $a\mar 1=1$;
\item[$(PBCK4)$] $a\ra b=1=b\ra a=1$ ($a\mar b=1=b\mar a=1$) implies that $a=b$.

For any pseudo $BCK$-algebra one can define a partially order relation $\leq$ by $a\leq b$ if and only if $a\ra b=1$ (if and only if $a\mar b=1)$.
Any pseudo $BCK$-algebra   satisfies the following conditions (see \cite{Gor,Ior}):
\begin{itemize}
\item[(P1)] $x\leq y$ implies that $z\ra x\leq z\ra y$ and $z\mar x\leq z\mar y$;
\item[(P2)] $x\leq y$ implies that $y\ra z\leq x\ra z$ and $y\mar z\leq x\mar z$;
\item[(P3)] $x\ra y\leq (z\ra x)\ra (z\ra y)$;
\item[(P4)] $x\mar y\leq (z\mar x)\ra (z\mar y)$.
\end{itemize}
A pseudo $BCK$-{\it meet-semilattice} is an algebra $(X;\ra,\mar,\wedge,1)$ of
type $(2,2,2,0)$ such that $(X;\ra,\mar,1)$ is a pseudo $BCK$-algebra and
its underlying partial order implies that $(X;\wedge)$ is a meet-semilattice.
\end{defn}

We note that the class of pseudo $BCK$-algebras does not form a variety because it is not closed under homomorphic images, see e.g. \cite[Thm 1.4]{MeJu}. On the other hand, the class of pseudo $BCK$-meet-semilattices does: An algebra $(X;\ra,\mar,\wedge,1)$ of type $(2,2,2,1)$ is a pseudo $BCK$-meet-semilattice, \cite[p. 6]{Kuh}, if and only if it satisfies  $(PBCK1)-(PBCK3)$ and

\begin{itemize}
\item[$(SL1)$] $a\wedge [(a\ra b)\mar b]= a$;
\item[$(SL2)$] $(a\wedge b)\ra b =1$.
\end{itemize}

\begin{exm}\label{ex:pBCK}
Let $(G;\cdot,^{-1},e,\le)$ be an $\ell$-group. If we endow $G^-$ with two binary operations $a\ra b = (ba^{-1})\wedge e$ and $a \mar b = (a^{-1}b)\wedge e$, then $(G^-;\ra,\mar,e)$ is a pseudo BCK-algebra that is even a lattice; cf. Example \ref{ex:l-group}.
\end{exm}

\begin{thm}\label{4.2}
\begin{enumerate}
\item[{\rm (i)}] Let $(X;\sim,\bsim,\wedge,1)$ be a pseudo equality algebra.
Then $(X;\ra,\mar,\wedge,1)$ is a pseudo $BCK$-meet-semilattice satisfying
condition {\rm (xii)} of Proposition {\rm \ref{3.5}}, where
$a\ra b=(a\wedge b)\sim a$ and $a\mar b=a\bsim(a\wedge b)$ for all $a,b\in X$.
\vspace{-2mm}
\item[{\rm(ii)}] Let $(X;\ra,\mar,\wedge,1)$ be a pseudo $BCK$-meet-semilattice
satisfying  condition {\rm (xii)} of Proposition {\rm \ref{3.5}}.
Then $(X;\sim,\bsim,\wedge,1)$ is a pseudo equality algebra, where $a\sim b=b\ra a$ and $a\bsim b=a\mar b$ for all $a,b\in X$.
\end{enumerate}
\end{thm}

\begin{proof}
(i) The proof follows from Proposition \ref{3.5}.

(ii) Let $a,b,c\in X$. Clearly, $(F1)$, $(F2)$ and $(F3)$ hold. $(F4)$ follows from (P1) and (P2).
$(F5)$ is straightforward by  our assumption. Also, $(PBCK1)$ implies $(F6)$ and (P3) and (P4) imply $(F7)$,
so $(X;\sim,\bsim,\wedge,1)$ is a pseudo equality algebra.
\end{proof}

\begin{rmk}\label{4.3}
If $(X;\ra,\mar,\wedge,1)$ is a pseudo $BCK$-meet-semilattice satisfying condition (xii) of Proposition \ref{3.5}, then for each $a,b\in X$,
$a\ra b\leq (a\wedge a)\ra (a\wedge b)=a\ra (a\wedge b)$. Also, by (P1),
$a\ra (a\wedge b)\leq a\ra b$, so $a\ra (a\wedge b)= a\ra b$. In a similar way,
we can show that $a\mar (a\wedge b)=a\mar b$.
\end{rmk}

By Theorem \ref{4.2}, if $(X;\sim,\bsim,\wedge,1)$ is a pseudo equality algebra, then
$F((X;\sim,\bsim,\wedge,1)):=(X;\ra,\mar,\wedge,1)$ is a pseudo $BCK$-meet-semilattice, where
$a\ra b=(a\wedge b)\sim b$ and $a\mar b=a\bsim (a\wedge b)$ for all $a,b\in X$. Moreover, if
$(X;\ra,\mar,\wedge,1)$ is a $BCK$-meet-semilattice that satisfies
the condition (xii) of Proposition \ref{3.5}, then
$G((X;\ra,\mar,\wedge,1)):=(X;\sim',\bsim',\wedge,1)$ is a pseudo equality algebra, where
$a\sim' b=b\ra a$ and $a\backsim' b=a\rightsquigarrow b$ for all $a,b\in X$.
The category whose objects are pseudo equality algebras and whose morphisms are
homomorphisms of pseudo equality algebras is called the category of pseudo equality algebras and is
denoted by $\mathcal{A}$. The category of pseudo $BCK$-meet-semilattices can be defined similarly. Let
$\mathcal{B}$ be its subcategory whose objects are pseudo $BCK$-meet-semilattices
satisfying condition (xii) of Proposition \ref{3.5}. Then clearly, $F:\mathcal{A}\ra \mathcal{B}$ and
$G:\mathcal{B}\ra \mathcal{A}$ are functors. In the next theorem we want to verify a relation between
these functors.

\begin{defn}\label{4.4.0}
A pseudo equality algebra $(X;\sim,\bsim,\wedge,1)$ is called {\it invariant} if
there exists a pseudo $BCK$-meet-semilattice $(X;\mapsto,\rightarrowtail,\wedge,1)$ such that
$G((X;\mapsto,\rightarrowtail,\wedge,1))=(X;\sim,\bsim,\wedge,1)$.
\end{defn}

\begin{thm}\label{4.4}

\begin{enumerate}
\item[{\rm (i)}] Let $(X;\sim,\bsim,\wedge,1)$ be a pseudo equality algebra. Then
$F(G(F((X;\sim,\bsim,\wedge,1))))=F((X;\sim,\bsim,\wedge,1))$.

\vspace{-2mm}
\item[{\rm (ii)}] Let $(X;\ra,\mar,\wedge,1)$ be a pseudo $BCK$-meet-semilattices
satisfying condition {\rm (xii)} of Proposition {\rm \ref{3.5}}. Then
$F(G((X;\ra,\mar,\wedge,1)))=(X;\ra,\mar,\wedge,1)$.

\vspace{-2mm}
\item[{\rm(iii)}]  A pseudo equality algebra $(X;\sim,\bsim,\wedge,1)$ is {\it invariant} if and only if
$GF((X;\sim,\bsim,\wedge,1))=(X;\sim,\bsim,\wedge,1)$.

\vspace{-2mm}
\item[{\rm(iv)}] The class of pseudo $BCK$-meet-semilattices
satisfying condition {\rm (xii)} of Proposition {\rm \ref{3.5}}, $Obj(\mathcal{B})$,  and the class of invariant pseudo equality algebras are term equivalent.

\vspace{-2mm}
\item[{\rm(v)}] The category $Inv({\mathcal A})$ of invariant pseudo equality algebras and the category $\mathcal{B}$ are
categorically equivalent.
\end{enumerate}
\end{thm}

\begin{proof}
(i) Let $\ra'$ and $\mar'$ be the binary operations derived by $\sim'$ and $\bsim'$ on the
pseudo equality algebra $G(F((X;\sim,\bsim,\wedge,1)))$, respectively. It suffices to show that
$\ra'=\ra$ and $\mar=\mar'$. Let $a,b\in X$. By definitions of $\ra'$ and $\mar'$ we have
$a\ra' b=(a\wedge b)\sim' a=a\ra (a\wedge b)=(a\wedge b\wedge a)\sim a=(a\wedge b)\sim a=a\ra b$ and
$a\mar' b=a\bsim' (a\wedge b)=a\mar (a\wedge b)=a\bsim(a\wedge b\wedge a)=a\bsim (a\wedge b)=a\mar b$.

(ii)
Let $\sim'$ and $\bsim'$ be the binary operations induced by $\ra$ and $\mar$ on the pseudo $BCK$-meet-semilattice
$(X;\ra,\mar,\wedge,1)$ and $\ra'$ and $\mar'$ be two derived operations on the pseudo equality algebra  $(X;\sim',\bsim',\wedge,1)=G((X;\ra,\mar,\wedge,1))$. By definition, we know that
$F(G((X;\ra,\mar,\wedge,1)))=G((X;\ra',\mar',\wedge,1))$. It suffices to show that $\ra=\ra'$ and
$\mar=\mar'$. Put $a,b\in X$.
$a\ra' b=(a\wedge b)\sim' a=a\ra (a\wedge b)=a\ra b$ (by Remark \ref{4.3}) and
$a\mar' b=a\bsim' (a\wedge b)=a\mar (a\wedge b)=a\mar b$. Therefore,
$F(G((X;\ra,\mar,\wedge,1)))=(X;\ra,\mar,\wedge,1)$.

(iii) The proof follows from (i) and (ii).

(iv) Let $Inv(\mathcal{A})$ be the class of invariant pseudo equality algebras.
By (ii), the map $F:Inv(\mathcal A)\ra Obj(\mathcal{B})$ is onto. Also, if
$(Y,+,-,\wedge,1)$ and $(Y,\sim,\bsim,\wedge,1)$
are two invariant pseudo equality algebras such that, $F((X;\sim,\bsim,\wedge,1))=F((Y,+,-,\wedge,1))$, then
by (iii),
$(X;\sim,\bsim,\wedge,1)=G(F((X;\sim,\bsim,\wedge,1)))=G(F((Y,+,-,\wedge,1)))=(Y,+,-,\wedge,1)$. Therefore, $F$ is a one-to-one map.

(v) Straightforward.
\end{proof}

Let take a pseudo $BCK$-meet-semilattice from \cite{Ciu} given by the following tables:

$$
\begin{matrix}
 & \ra & \vline & 0 & a & b & c & 1\\
\hline &0    & \vline & 1 & 1 & 1 & 1 & 1\\
& a & \vline & 0 & 1 & b & 1 & 1\\
& b & \vline & a & a & 1 & 1 & 1\\
& c & \vline & 0 & a & b & 1 & 1\\
& 1 & \vline & 0 & a & b & c & 1
\end{matrix}
$$

$$
\begin{matrix}
 & \mar & \vline & 0 & a & b & c & 1\\
\hline &0    & \vline & 1 & 1 & 1 & 1 & 1\\
& a & \vline & b & 1 & b & 1 & 1\\
& b & \vline & 0 & a & 1 & 1 & 1\\
& c & \vline & 0 & a & b & 1 & 1\\
& 1 & \vline & 0 & a & b & c & 1
\end{matrix}
$$

Then $c \ra b = b$ and $b \not\leq(c\wedge a)\ra (b\wedge a)= 0$, so that this pseudo $BCK$-meet-semilattice does not satisfy (xii) of Proposition \ref{3.5}.

\section{Congruences and deductive systems on pseudo equality algebras }

We show that congruences on pseudo equality algebras are closely connected wit normal closed deductive systems and we describe the lattice of these congruences.

\begin{defn}\label{5.1}
Let $F$ be a subset of a pseudo equality algebra $(X;\sim,\bsim,\wedge,1)$ containing $1$. Then
$F$ is called a
\begin{itemize}
\item {\it $(\sim,\bsim)$-deductive system } if
\subitem(a) $F$ is an upset. That is, $a\in F$ and $a\leq b\in X$ imply that $b\in F$;
\subitem(b) $a,b\sim a\in F$ imply that $b\in F$ for all $a,b\in X$.
\item {\it $(\ra,\mar)$-deductive system} if, for all $a,b\in X$, $a,a\ra b\in F$ imply that $b\in F$.
\end{itemize}
\end{defn}

\begin{lem}\label{le:6.1} Let $F$ be a subset of a pseudo equality algebra $(X;\sim,\bsim,\wedge,1)$
containing $1$.
\begin{itemize}
\item[{\rm (1)}] If $F$ is a $(\ra,\mar)$-deductive system, then $a\in F$ and $a\le b \in X$ imply $b \in F$.
\item[{\rm (2)}] $F$ is a $(\ra,\mar)$-deductive system if and only if $a,a\mar b\in F$ imply $b\in F$.
\item[{\rm (3)}] $F$ is a $(\sim,\bsim)$-deductive system if and only if  $F$ is an upset such that
$a,a\bsim b\in F$ imply $b\in F$.
\end{itemize}
\end{lem}
\begin{proof}
(1) Let $a \le b$ and $a \in F$. By (v) of Proposition \ref{3.5}, we have $a\ra b =1$. Then $a\ra b =1 \in F$, so that $b\in F$.

(2) Let $F$ be a $(\ra,\mar)$-deductive system and let $a,a\mar b\in F$. By (viii) of Proposition \ref{3.5}, we have $a \le (a\mar b)\ra b\in F$ which yields $b \in F$.

Conversely, let $a,a\ra b\in F$. Then $a \le (a\ra b)\mar b \in F$ which entails $b \in F$ and $F$ is a $(\ra,\mar)$-deductive system.

(3) Let $F$ be a $(\sim,\bsim)$-deductive system and let $a,a\bsim b\in F$. By Proposition \ref{3.5}(ii),
$a\leq b\sim (a\bsim b)$, so $b\in F$. Conversely, let $a,b\sim a\in F$. Then by Proposition \ref{3.5}(ii),
$a\leq (b\sim a)\bsim b$, hence $(b\sim a)\bsim b\in F$ and so $b\in F$.
\end{proof}

\begin{prop}\label{5.2}
Let $F$ be a non-empty subset of a pseudo equality algebra $(X;\sim,\bsim,\wedge,1)$. Then
$F$ is a $(\sim,\bsim)$-deductive system if and only if $F$ is a $(\ra,\mar)$-deductive system.
\end{prop}

\begin{proof}
Let $F$ be a $(\sim,\bsim)$-deductive system. If $x,x\ra y\in F$ for some $x,y\in X$,
then $x,(x\wedge y)\sim x\in F$, hence by the assumption, $x\wedge y\in F$ and so
$y\in F$ (since $F$ is an upset). Therefore, $F$ is a $(\ra,\mar)$-deductive system.
Conversely, let $F$ be a $(\ra,\mar)$-deductive system. If $x,y\sim x\in F$, then by Proposition \ref{3.5}(i),
$x\ra y\in F$ (since $y\sim x\leq x\ra y$) which implies that $y\in F$.
Therefore, $F$ is a $(\sim,\bsim)$-deductive system.
\end{proof}

From now on, in this paper, since $(\sim,\bsim)$-deductive systems and $(\ra,\mar)$-deductive systems are
equivalent, we called them {\it deductive systems}, for short, and we use $DS(X)$ to
denote the set of all deductive systems of $(X;\sim,\bsim,\wedge,1)$.

In the following proposition we show that any deductive system of a pseudo equality algebra is closed under $\ra, \mar, \wedge$.

\begin{prop}\label{pr:property}
If $F$ is a deductive system of a pseudo equality algebra $(X;\sim,\bsim,\wedge,1)$ and $a,b\in F$, then $a\ra b, a\mar b, a\wedge b \in F$.
\end{prop}

\begin{proof}

(a) By (vii) of Proposition \ref{3.5}, we have $b\leq (a\rightarrow b)\wedge (a\rightsquigarrow b)$ so by Lemma \ref{le:6.1}(1),
$a\rightarrow b\in F$ and $a\rightsquigarrow b\in F$.

(b) By (xiii) of Proposition \ref{3.5} and (a), we have $a\ra b = a\bsim (a\wedge b) \in F$ which yields $a\wedge b \in F$.
\end{proof}

We do not know whether any deductive system is closed under $\sim$ and $\backsim$, therefore, we introduce the following notion:
A deductive system $F\in DS(X)$ is called {\it closed}
if $x\sim y,x\bsim y\in F$ for all $x,y\in F$. An equivalent property is the following statement:

\begin{prop}
A deductive system $F$ of a pseudo equality algebra
$(X;\sim,\bsim,\wedge,1)$ is closed if and only if $1\sim x,x\bsim 1\in F$ for all $x\in F$.
\end{prop}

\begin{proof}
Let $F$ be a deductive system of $X$ such that $1\sim x,x\bsim 1\in F$, for all
$x\in F$. Put $x,y\in F$. Then by $(F6)$,
$x\bsim (x\sim y)=(x\sim 1)\bsim (x\sim y)\geq 1\sim y$ and
$(x\bsim y)\sim y=(x\bsim y)\sim (1\bsim y)\geq x\bsim 1$, we have $x\bsim y,x\sim y\in F$ and so
$F$ is closed. The proof of the converse is clear.
\end{proof}

Using Theorem \ref{4.4}, we show that every deductive system of an invariant pseudo equality algebra is closed.

\begin{exm}
Let $(X;\sim,\backsim,\wedge,1)$ be an invariant pseudo equality algebra and $F$ be a deductive system of $X$. We assert that $F$ is closed.

Indeed, by Theorem \ref{4.4}(ii), $G(F((X;\sim,\backsim,\wedge,1)))=(X;\sim,\backsim,\wedge,1)$.
For all $x,y\in X$, we have $x\sim' y=x\sim y$ and $x\backsim' y=x\backsim y$, where $\sim'$ and $\backsim'$ are binary operations induced by $\rightarrow$ and $\rightsquigarrow$, i.e., $x \sim' y:=y\rightarrow x$ and $x\backsim' y:= x\mar y$, in the
pseudo equality algebra $F((X;\sim,\backsim,\wedge,1))=(X;\ra,\mar,\wedge,1)$. It follows that
$x\sim y=x\sim' y=y\rightarrow x$ and $x\backsim y=x\backsim' y=x\rightsquigarrow y$, for all $x,y\in X$.

By Proposition \ref{pr:property}, $F$ is closed under $\sim$ and $\backsim$.

\end{exm}

The following notion will enable us to study congruences via normal closed deductive systems.

\begin{defn}\label{5.3}
A deductive system $F$ of a pseudo equality algebra
$(X;\sim,\bsim,\wedge,1)$ is called {\it normal}
if, for all $x,y\in X$, we have: $x\sim y, y\sim x \in F \Longleftrightarrow y\bsim x, x\bsim y \in F$.

\end{defn}



An equivalence relation $\theta$ on a pseudo equality algebra $(X;\sim,\bsim,\wedge,1)$ is called
a {\it congruence relation} if, for all $*\in\{\sim,\bsim,\wedge\}$ and all $(a,b),(x,y)\in \theta$,
$(a*x,b*y)\in \theta$. Denote by $Con(X)$ the set of all congruence relation on a pseudo equality algebra $(X;\sim,\bsim,\wedge,1)$.

\begin{prop}\label{5.4}
If $\theta$ is a congruence relation on a pseudo equality algebra $(X;\sim,\bsim,\wedge,1)$, then
$F_\theta=[1]_\theta=\{x\in X|\ (x,1)\in\theta\}$ is a closed normal deductive system of $X$.
\end{prop}

\begin{proof}
First we show that $F_\theta$ is a deductive system. Clearly, $1\in F_\theta$. Let $y\sim x,x\in F_\theta$. Then
$(x,1)\in\theta$, so $(y\sim x,y)=(y\sim x,y\sim 1)\in\theta$ and hence $(1,y)\in\theta$. Thus,
$y\in F_\theta$.
Suppose that $x\leq y$, $x\in F_\theta$ and $y\in X$. Then
$$
(1,x)\in\theta\Rightarrow (x\wedge y,1\wedge y)\in\theta\Rightarrow (x,y)\in\theta\Rightarrow (1,y)\in\theta \Rightarrow y\in F_\theta
$$
and so $F_\theta$ is a deductive system. Now, we show that $F_\theta$ is normal.
Let $x\sim y, y\sim x\in F_\theta$ for some $x,y\in X$. From $(x\sim y,1)\in \theta$ and  Proposition \ref{3.5}(ii), it follows that
$y=y\wedge((x\sim y)\bsim x)$, $y\wedge((x\sim y)\bsim x)\theta y\wedge (1\bsim x)$ and
$ y\wedge (1\bsim x)=y\wedge x$ giving $(y,y\wedge x)\in\theta$.
Similarly, $(y\sim x,1)\in \theta$ implies that $(x,x\wedge y)\in\theta$ and so $(x,y)\in \theta$.
Hence, $(1,y\bsim x)=(x\bsim x,y\bsim x)\in\theta$ and $(x\bsim y,1)=(x\bsim y,y\bsim y)\in\theta$, whence
$y\bsim x, x\bsim y\in F_\theta$. Conversely, assume $y\bsim x, x\bsim y\in F_\theta$. In a similar way we can show that $x\sim y, y\sim x\in F_\theta$ which proves $F_\theta$ is normal.
%
\end{proof}

\begin{prop}\label{5.5}
Let $F$ be a deductive system of a pseudo equality algebra $(X;\sim,\bsim,\wedge,1)$.
\begin{itemize}
\item[{\rm (i)}] The relation
$\theta_F=\{(x,y)\in X\times X| \ x\sim y,y\sim x, x\bsim y, y\bsim x\in F\}$
is an equivalence relation on $X$.
\item[{\rm (ii)}] If $F$ is normal, then
$\theta_F=\{(x,y)\in X\times X| \ x\sim y,y\sim x\in F\}=
\{(x,y)\in X\times X| \ x\bsim y,y\bsim x\in F\}$ is a congruence relation on $X$.
\end{itemize}
\end{prop}

\begin{proof}
(i) Clearly, $\theta_F$ is reflexive and symmetric. Let $(x,y),(y,z)\in \theta_F$.
Then by $(F6)$, $x\bsim y\leq (x\bsim z)\sim (y\bsim z)$, so $(x\bsim z)\sim (y\bsim z)\in F$
(since $x\bsim y\in F$). Using $y\bsim z\in F$, we get that $x\bsim z\in F$. Moreover,
from $z\bsim y\leq (z\bsim x)\sim (y\bsim x)$, $z\bsim y\in F$ and $y\bsim x\in F$ we get that
$z\bsim x\in F$. In a similar way, we can show that $x\sim z,z\sim x\in F$.
Thus $(x,z)\in\theta_F$ and so
$\theta_F$ is an equivalence relation on $X$.

(ii) Since $F$ is normal, clearly, $\theta_F=\{(x,y)\in X\times X| \ x\sim y,y\sim x\in F\}=
\{(x,y)\in X\times X| \ x\bsim y,y\bsim x\in F\}$. Put $(a,b),(x,y)\in\theta_F$.

(1) By $(F5)$, $b\sim a\leq (b\wedge x)\sim (a\wedge x)$,
$a\sim b\leq (a\wedge x)\sim (b\wedge x)$, so $(a\wedge x,b\wedge x)\in\theta_F$. Similarly,
$(b\wedge x,b\wedge y)\in\theta_F$. From (i), it follows that $(a\wedge x,b\wedge y)\in\theta_F$.

(2) By $(F6)$, we have $x\bsim y\leq (x\bsim a)\sim (y\bsim a)$ and  $y\bsim x\leq (y\bsim a)\sim (x\bsim a)$,
hence $(x\bsim a)\sim (y\bsim a),(y\bsim a)\sim (x\bsim a)\in F$ and so
$(x\bsim a,y\bsim a)\in\theta_F$. Moreover, by $(F7)$,
$a\bsim b\leq (y\bsim a)\bsim (y\bsim b)$ and $b\bsim a\leq (y\bsim b)\bsim (y\bsim a)$, so
$(y\bsim a,y\bsim b)\in\theta_F$. Therefore,  $(x\bsim a,y\bsim a)\in\theta_F$ and
$(y\bsim a,y\bsim b)\in\theta_F$, and by (i), we obtain $(x\bsim a,y\bsim b)\in\theta_F$.
In a similar way, we can prove that $(x\sim a,y\sim b)\in\theta_F$. Consequently,
$\theta_F$ is a congruence relation on $X$.
\end{proof}

\begin{thm}\label{5.6}
Let $(X;\sim,\bsim,\wedge,1)$ be a pseudo equality algebra. Then
there is a one-to-one correspondence between the set of all normal closed deductive systems of $X$, $NCDS(X)$, and $Con(X)$.
\end{thm}

\begin{proof}
Let $\phi: Con(X)\ra NCDS(X)$ be a mapping defined by $\phi(\theta)=F_\theta$ for all $\theta\in Con(X)$
(by Proposition \ref{5.4} it is well defined).
Let $\theta$ be a congruence relation on $(X;\sim,\bsim,\wedge,1)$. It is clear that
$\theta\s \theta_{F_{\theta}}$. Put $(x,y)\in\theta_{F_{\theta}}$.
Then $x\sim y,y\sim x\in F_\theta$, so $(x\sim y,1),(y\sim x,1)\in\theta$.
Similar to the proof of
Proposition \ref{5.4}, we have $(x,x\wedge y),(y,x\wedge y)\in\theta$.
It follows that $(x,y)\in\theta$, which yields $\theta=\theta_{F_{\theta}}$.
If $\theta,\theta'\in Con(X)$ such that $\phi(\theta)=\phi(\theta')$,
then $\theta=\theta_{F_{\theta}}=\theta_{F_{\theta'}}=\theta'$.
Now, let $F$ be a normal closed deductive system of $X$. Since $F$ is closed, then
$$
x\in [1]_{\theta_{F}} \Leftrightarrow (x,1)\in\theta_F\Leftrightarrow x\sim 1, 1\sim x\in F \Leftrightarrow x\in F.
$$
and hence $F=\phi(\theta_F)$. By summing up the above results, we get that $\phi$ is a one-to-one correspondence.
\end{proof}

\begin{thm}\label{thm:5.7}
Let $(X;\sim,\bsim,\wedge,1)$ be a pseudo equality algebra and $F$ be a normal deductive system of $X$. Then $(X/F;\simeq,\backsimeq,\underline{\wedge},1/F)$, where
$X/F=\{x/F|\ x\in X\}$ with elements $x/F=\{y\in X|\ (x,y)\in \theta_F\}$ is endowed with binary operations $\simeq,\backsimeq,\underline{\wedge}$ and with a nullary operation $1/F$, is a pseudo equality algebra, where
$$
x/F\simeq y/F:=(x\sim y)/F,\quad x/F\backsimeq y/F:=(x\bsim y)/F,\quad x/F\underline{\wedge}\ y/F:=(x\wedge y)/F.
$$
\end{thm}

\begin{proof}
Clearly, $(X/F;\underline{\wedge},1/F)$ is a meet-semilattice with top element $1/F$ and
$x/F\leq y/F \Leftrightarrow x/F\underline{\wedge}\ y/F=x/F\Leftrightarrow x/F=(x\wedge y)/F\Leftrightarrow x\wedge y\in x/F$.
Moreover, it satisfies the conditions
$(F2)$, $(F3)$, $(F5)-(F7)$ hold.
Let $a/F\leq b/F\leq c/F$, for some $a,b,c\in X$. Then $a\wedge b\in a/F$ and $b\wedge c\in b/F$, hence
$a\wedge (b\wedge c)\theta_F a\wedge b\theta_F a$. Set $u:=a\wedge b\wedge c$, $v=b\wedge c$ and $w=c$. Then
$u\in a/F$, $v\in b/F$, $w\in c/F$ and $u\leq v\leq w$. By $(F3)$, we get that
$u\sim w\leq v\sim w$, $u\bsim w\leq u\bsim v$, $w\sim u\leq w\sim v$ and $w\bsim u\leq v\bsim a$.
\begin{eqnarray*}
u\sim w\leq v\sim w\Rightarrow u\sim w=(u\sim w)\wedge (v\sim w)
\Rightarrow u/F\simeq w/F=(u\sim w)/F=((u\sim w)\wedge (v\sim w))/F\\
=(u/F\simeq w/F)\underline{\wedge} (v/F\simeq w/F)\Rightarrow u/F\simeq w/F\leq v/F\simeq w/F\Rightarrow
a/F\simeq c/F\leq b/F\simeq c/F.
\end{eqnarray*}
In a similar way, it can be proved that $a/F\simeq c/F\leq a/F \simeq b/F$, $c/F\backsimeq a/F\leq c/F\backsimeq b/F$ and
$c/F\backsimeq a/F\leq b/F\backsimeq a/F$. Therefore,
$(X/F;\simeq,\backsimeq,\underline{\wedge},1/F)$ is a pseudo equality algebra.
\end{proof}

\begin{defn}
A deductive system $F$ of a pseudo equality algebra $(X;\sim,\bsim,\wedge,1)$ is called
{\it commutative} if $(a\sim b)\sim(b\sim a)\in F$ for all $a,b\in X$.
\end{defn}
In the next theorem, we use a normal commutative deductive system to obtain an equality algebra from a pseudo equality algebra.

\begin{prop}\label{prop:5.8}
Let $F$ be a normal commutative deductive system of a pseudo equality algebra $(X;\sim,\bsim,\wedge,1)$. Then
$(X/F;\simeq,\underline{\wedge},1/F)$ is an equality algebra.
\end{prop}

\begin{proof}
By Theorem \ref{thm:5.7}, $(X/F;\simeq,\backsimeq,\underline{\wedge},1/F)$ is a pseudo equality algebra. Put $a,b\in X$.
Then by the assumption, $(a\sim b)\sim (b\sim a),(b\sim a)\sim (b\sim a)\in F$ and so
$(a\sim b)/F=(b\sim a)/F$. By $(F6)$,
$a/F\backsimeq b/F\leq (a/F\backsimeq a/F)\simeq  (b/F\backsimeq a/F)=(1/F)\simeq  (b/F\backsimeq a/F)=(b/F\backsimeq a/F)\simeq (1/F)=
b/F\backsimeq a/F$. Similarly, $b/F\backsimeq a/F\leq a/F\backsimeq b/F$. Thus, $\backsimeq$ and $\simeq$ are commutative binary operations
on $X/F$. Therefore, Remark \ref{3.2} implies that $(X/F;\simeq,\underline{\wedge},1/F)$ is an equality algebra.
\end{proof}

Now we describe some properties of the lattice of congruences of a pseudo equality algebra.

\begin{defn}\cite{Ursini}
A variety is said to be {\it subtractive} if it satisfies the identities
$S(x,x)=0$ and $S(x,0)=x$ for some binary term $S$ and some nullary term $0$.
\end{defn}

\begin{thm}\label{subtractive}
The variety of pseudo equality algebras is subtractive.
\end{thm}

\begin{proof}
Putting $S(x,y)=x\sim y$ and $0=1$, the proof follows from $(F2)$ and $(F3)$.
\end{proof}

\begin{thm}\label{Permutable}
The variety of pseudo equality algebras is congruence permutable and
congruence distributive.
\end{thm}

\begin{proof}
Let $(X;\sim,\bsim,\wedge,1)$ be a pseudo equality algebra. Define
two terms $M(x,y,z):=((x\sim y)\bsim z)\wedge ((y\sim z)\bsim x)$ and
$P(x,y,z):=((x\ra y)\mar y)\wedge ((y\ra z)\mar z)\wedge ((z\ra x)\mar x)$ for all $x,y,z\in X$. We show that they are Mal'cev terms, that is, they satisfy identities $M(x,y,y)=M(y,y,x)=x$ and $P(x,x,y)=P(x,y,x)=P(y,x,x)=x$.

By $(F2)$, $(F3)$ and Proposition \ref{3.5}(ii), we have
$M(x,y,y)= ((x\sim y)\bsim y)\wedge ((y\sim y)\bsim x)=((x\sim y)\bsim y)\wedge x=x$ and
$M(y,y,x)=((y\sim y)\bsim x)\wedge ((y\sim x)\bsim y)=x\wedge ((y\sim x)\bsim y)=x$, for all
$x,y\in X$. Moreover, by Proposition \ref{3.5}(vi) and (viii),
$P(x,x,y)=((x\ra x)\mar x)\wedge ((x\ra y)\mar y)\wedge ((y\ra x)\mar x)=x\wedge ((x\ra y)\mar y)\wedge ((y\ra x)\mar x)=x$ and by Proposition \ref{3.5}(vii),
$P(x,y,x)=((x\ra y)\mar y)\wedge ((y\ra x)\mar x)\wedge ((x\ra x)\mar x)=x$
and $P(y,x,x)=((y\ra x)\mar x)\wedge ((x\ra x)\mar x)\wedge ((x\ra y)\mar y)=x$. Therefore, by
\cite[Lemma 1.24]{GJKO}, the theorem is proved.
\end{proof}

\section{Conclusion}

We have showed that pseudo equality algebras in the sense of \cite{JK} are always equality algebras in the sense of \cite{J}. Therefore, we have presented a new version of pseudo equality algebras which generalize equality algebras which could be assumed for a possible algebraic semantics of fuzzy type theory.

We have described their basic properties, presented examples of pseudo equality algebras. We pointed out a close relation between invariant pseudo BCK-algebras and pseudo BCK-meet-semilattices with a special condition, Theorem \ref{4.4}. This result corresponds in some sense with ideas of Kabzi\'nski and Wro\'nski \cite{KaWr} on the equivalent algebraic semantics of the $\leftrightarrow$-fragment of intuitionistic logic.

We describe congruences equivalently via normal closed deductive systems, Theorem \ref{5.6} and we show that the variety of pseudo equality algebras is subtractive, Theorem \ref{subtractive}, and congruence permutable and congruence distributive, Theorem \ref{Permutable}.

These results may be assumed as an additional step in establishing an axiomatization of special subclasses of substructural logics.

\end{document}